\documentclass[a4paper,11pt,leqno]{amsart}
\usepackage[cp1250]{inputenc}
\usepackage{ifthen}
\usepackage{amssymb}
\usepackage{amsmath}
\usepackage{amsfonts}
\usepackage{latexsym}

\numberwithin{equation}{section}
\newtheorem{theorem}{Theorem}[section]
\theoremstyle{plain}

\newtheorem{corollary}[theorem]{Corollary}

\newtheorem{definition}{Definition}[section]

\newtheorem{lemma}{Lemma}[section]

\numberwithin{equation}{section}
\input{tcilatex}

\begin{document}
\title[Coefficient bounds for bi-univalent functions]{Coefficient bounds for
new subclasses of bi-univalent functions }
\author{Murat ÇA\u{G}LAR}
\address{Department of Mathematics, Faculty of Science, Ataturk University,
25240 Erzurum, Turkey.}
\email{mcaglar@atauni.edu.tr (Murat Ça\u{g}lar); orhan@atauni.edu.tr (Halit
Orhan)}
\author{Halit ORHAN}
\author{Nihat YA\u{G}MUR}
\address{Department of Mathematics, Faculty of Science and Arts, Erzincan
University, Erzincan, Turkey.}
\email{nhtyagmur@gmail.com (Nihat Yagmur)}

\begin{abstract}
In the present investigation, we consider two new subclasses $\mathcal{N}%
_{\Sigma }^{\mu }\left( \alpha ,\lambda \right) $ and $\mathcal{N}_{\Sigma
}^{\mu }\left( \beta ,\lambda \right) $ of bi-univalent functions defined in
the open unit disk $\mathcal{U}=\left\{ {z:\left\vert z\right\vert <1}%
\right\} .$ Besides, we find upper bounds for the second and third
coefficients for functions in these new subclasses.
\end{abstract}

\maketitle

\section{Introduction and definitions}

Let $\mathcal{A}$ denote the family of functions $f$ \ of the form%
\begin{equation}
f(z)=z+\sum\limits_{n=2}^{\infty }{a_{n}z^{n}}  \label{eq1}
\end{equation}%
which are analytic in the open unit disk $\mathcal{U}=\left\{ {z:\left\vert
z\right\vert <1}\right\} $. Further, let $\mathcal{S}$ denote the class of
functions which are univalent in $\mathcal{U}$.

It is well known that every function $f\in \mathcal{S}$ has an inverse $%
f^{-1},$ defined by%
\begin{equation*}
f^{-1}(f(z))=z\text{ \ \ }\left( z\in \mathcal{U}\right)
\end{equation*}%
and%
\begin{equation*}
f(f^{-1}(w))=w\text{ \ \ }\left( \left\vert w\right\vert <r_{0}(f);\text{ }%
r_{0}(f)\geq \frac{1}{4}\text{\ }\right)
\end{equation*}%
where%
\begin{equation*}
f^{-1}(w)=w-a_{2}w^{2}+(2a_{2}^{2}-a_{3})w^{3}-(5a_{2}^{3}-5a_{2}a_{3}+a_{4})w^{4}+....
\end{equation*}

A function $f\in \mathcal{A}$ is said to be bi-univalent in $\mathcal{U}$ if
both $f(z)$ and $f^{-1}(z)$ are univalent in $\mathcal{U}.$

Let $\Sigma $ denote the class of bi-univalent functions in $\mathcal{U}$
given by (\ref{eq1}).

Lewin \cite{Lew} investigated the class $\Sigma $ of bi-univalent functions
and showed that\ $\left\vert a_{2}\right\vert <1.51$ for the functions
belonging to $\Sigma .$ Subsequently, Brannan and Clunie \cite{Bran2}
conjectured that $\left\vert a_{2}\right\vert \leqq \sqrt{2}.$ On the other
hand, Netanyahu \cite{Net} showed that $\underset{f\in \Sigma }{\max }%
\left\vert a_{2}\right\vert =\frac{4}{3}.$ The coefficient estimate problem
for each of $\left\vert a_{n}\right\vert $ $\left( n\in 
\mathbb{N}
\setminus \left\{ 1,2\right\} ;\text{ }%
\mathbb{N}
:=\left\{ 1,2,...\right\} \right) $ is still an open problem.

Brannan and Taha \cite{Bran} (see also \cite{Taha}) introduced certain
subclasses of the bi-univalent function class $\Sigma $ similar to the
familiar subclasses $\mathcal{S}^{\ast }(\alpha )$ and $\mathcal{K}(\alpha )$
of starlike and convex functions of order $\alpha (0<\alpha \leq 1)$,
respectively (see \cite{Bran1}). Thus, following Brannan and Taha \cite{Bran}
(see also \cite{Taha}), a function $f\in \mathcal{A}$ is in the class $%
\mathcal{S}_{\Sigma }^{\ast }[\alpha ]$ of strongly bi-starlike functions of
order $\alpha (0<\alpha \leq 1)$ if each of the following conditions is
satisfied:%
\begin{equation*}
f\in \Sigma \text{ and }\left\vert \arg \left( \frac{zf^{\prime }(z)}{f(z)}%
\right) \right\vert <\frac{\alpha \pi }{2}\text{ \ \ }(0<\alpha \leq 1,\text{
}z\in \mathcal{U})
\end{equation*}%
and 
\begin{equation*}
\left\vert \arg \left( \frac{zg^{\prime }(w)}{g(w)}\right) \right\vert <%
\frac{\alpha \pi }{2}\text{ \ \ }(0<\alpha \leq 1,\text{ }z\in \mathcal{U}),
\end{equation*}%
where g is the extension of $f^{-1}$ to $\mathcal{U}$. The classes $\mathcal{%
S}_{\Sigma }^{\ast }\left( \alpha \right) $ and $\mathcal{K}_{\Sigma
}(\alpha )$ of bi-starlike functions of order $\alpha $ and bi-convex
functions of order $\alpha ,$ corresponding, respectively, to the function
classes $\mathcal{S}_{\Sigma }^{\ast }\left( \alpha \right) $ and $\mathcal{K%
}_{\Sigma }(\alpha ),$ were also introduced analogously. For each of the
function classes $\mathcal{S}_{\Sigma }^{\ast }\left( \alpha \right) $ and $%
\mathcal{K}_{\Sigma }(\alpha )$, they found non-sharp estimates on the first
two Taylor--Maclaurin coefficients $|a_{2}|$ and $|a_{3}|$ (for details, see
(\cite{Bran}, \cite{Taha})).

Recently, Srivastava et al. \cite{Sri} and Frasin and Aouf \cite{Fra} have
investigated estimate on the coefficients $|a_{2}|$ and $|a_{3}|$ for
functions in the subclasses $\mathcal{N}_{\Sigma }^{1}\left( \alpha
,1\right) $, $\mathcal{N}_{\Sigma }^{1}\left( \alpha ,\lambda \right) $ and $%
\mathcal{N}_{\Sigma }^{1}\left( \beta ,1\right) ,\;\mathcal{N}_{\Sigma
}^{1}\left( \beta ,\lambda \right) $ which are given by Section 2 and 3,
respectively. The main object of the present investigation is to introduce
two new subclasses $\mathcal{N}_{\Sigma }^{\mu }\left( \alpha ,\lambda
\right) $ and $\mathcal{N}_{\Sigma }^{\mu }\left( \beta ,\lambda \right) $
of the function class $\Sigma $ and to find estimate on the coefficients $%
|a_{2}|$ and $|a_{3}|$ for functions in these new subclasses of the function
class $\Sigma $ employing the techniques used earlier by Srivastava et al. 
\cite{Sri}. We also extend and improve the aforementioned results of
Srivastava et al. \cite{Sri} and Frasin and Aouf \cite{Fra}. Various known
or new special cases of our results are also pointed out.

Firstly, in order to derive our main results, we need to following lemma.

\begin{lemma}
\label{l1}\cite{Pom} If $p\in \mathcal{P},$ then $\left\vert
c_{k}\right\vert \leq 2$ for each $k,$ where $\mathcal{P}$ is the family of
all functions $p$ analytic in $\mathcal{U}$ for which $\Re p(z)>0,$ $%
p(z)=1+c_{1}z+c_{2}z^{2}+...$ for $z\in \mathcal{U}.$
\end{lemma}

\section{Coefficient bounds for the function class $\mathcal{N}_{\Sigma }^{%
\protect\mu }\left( \protect\alpha ,\protect\lambda \right) $}

\begin{definition}
\label{d1} A function $f(z)$ given by (\ref{eq1}) is said to be in the class 
$\mathcal{N}_{\Sigma }^{\mu }\left( \alpha ,\lambda \right) $ if the
following conditions are satisfied:%
\begin{equation}
f\in \Sigma \text{ and }\left\vert \arg \left( \left( 1-\lambda \right)
\left( \frac{f(z)}{z}\right) ^{\mu }+\lambda f^{\prime }(z)\left( \frac{f(z)%
}{z}\right) ^{\mu -1}\right) \right\vert <\frac{\alpha \pi }{2}  \label{eq2}
\end{equation}%
\begin{equation*}
(0<\alpha \leq 1,\text{ }\mu \geq 0,\text{ }z\in \mathcal{U})
\end{equation*}%
and%
\begin{equation}
\left\vert \arg \left( \left( 1-\lambda \right) \left( \frac{g(w)}{w}\right)
^{\mu }+\lambda g^{\prime }(w)\left( \frac{g(w)}{w}\right) ^{\mu -1}\right)
\right\vert <\frac{\alpha \pi }{2}  \label{eq3}
\end{equation}%
\begin{equation*}
(0<\alpha \leq 1,\text{ }\mu \geq 0,\text{ }w\in \mathcal{U})
\end{equation*}%
where the function $g$ is given by%
\begin{equation}
g(w)=w-a_{2}w^{2}+(2a_{2}^{2}-a_{3})w^{3}-(5a_{2}^{3}-5a_{2}a_{3}+a_{4})w^{4}+....
\label{eq4}
\end{equation}
\end{definition}

Note that for $\lambda =\mu =1,$ the class $\mathcal{N}_{\Sigma }^{1}\left(
\alpha ,1\right) $ introduced and studied by Srivastava et al. \cite{Sri}
and for $\mu =1,$ the class $\mathcal{N}_{\Sigma }^{1}\left( \alpha ,\lambda
\right) $ introduced and studied by Frasin and Aouf \cite{Fra}.

\begin{theorem}
\label{t1}Let $f(z)$ given by (\ref{eq1}) be in the class $\mathcal{N}%
_{\Sigma }^{\mu }\left( \alpha ,\lambda \right) ,$ $0<\alpha \leq 1,$ $%
\lambda \geq 1$ and $\mu \geq 0.$ Then%
\begin{equation}
\left\vert a_{2}\right\vert \leq \frac{2\alpha }{\sqrt{\left( \lambda +\mu
\right) ^{2}+\alpha \left( \mu +2\lambda -\lambda ^{2}\right) }}  \label{eq5}
\end{equation}%
and%
\begin{equation*}
\left\vert a_{3}\right\vert \leq \frac{4\alpha ^{2}}{\left( \lambda +\mu
\right) ^{2}}+\frac{2\alpha }{2\lambda +\mu }.
\end{equation*}
\end{theorem}

\begin{proof}
It follows from (\ref{eq2}) and (\ref{eq3}) that%
\begin{equation}
\left( 1-\lambda \right) \left( \frac{f(z)}{z}\right) ^{\mu }+\lambda
f^{\prime }(z)\left( \frac{f(z)}{z}\right) ^{\mu -1}=\left[ p(z)\right]
^{\alpha }  \label{eq6}
\end{equation}%
and%
\begin{equation}
\left( 1-\lambda \right) \left( \frac{g(w)}{w}\right) ^{\mu }+\lambda
g^{\prime }(w)\left( \frac{g(w)}{w}\right) ^{\mu -1}=\left[ q(w)\right]
^{\alpha }  \label{eq7}
\end{equation}%
where $p(z)=1+p_{1}z+p_{2}z^{2}+...$ and $q(w)=1+q_{1}w+q_{2}w^{2}+...$ in $%
\mathcal{P}$.

Now, equating the coefficients in (\ref{eq6}) and (\ref{eq7}), we have%
\begin{eqnarray}
\left( \lambda +\mu \right) a_{2} &=&\alpha p_{1},  \label{eq8} \\
\left( 2\lambda +\mu \right) a_{3}+\left( \mu -1\right) \left( \lambda +%
\frac{\mu }{2}\right) a_{2}^{2} &=&\alpha p_{2}+\frac{\alpha \left( \alpha
-1\right) }{2}p_{1}^{2},  \label{eq9} \\
-\left( \lambda +\mu \right) a_{2} &=&\alpha q_{1}  \label{eq10}
\end{eqnarray}%
and%
\begin{equation}
-\left( 2\lambda +\mu \right) a_{3}+\left( 3+\mu \right) \left( \lambda +%
\frac{\mu }{2}\right) a_{2}^{2}=\alpha q_{2}+\frac{\alpha \left( \alpha
-1\right) }{2}q_{1}^{2}.  \label{eq11}
\end{equation}%
From (\ref{eq8}) and (\ref{eq10}), we obtain%
\begin{equation}
p_{1}=-q_{1}  \label{eq12}
\end{equation}%
and%
\begin{equation}
2\left( \lambda +\mu \right) ^{2}a_{2}^{2}=\alpha ^{2}\left(
p_{1}^{2}+q_{1}^{2}\right) .  \label{eq13}
\end{equation}%
Now, from (\ref{eq9}), (\ref{eq11}) and (\ref{eq13}), we get that%
\begin{eqnarray*}
\left( \mu +1\right) \left( 2\lambda +\mu \right) a_{2}^{2} &=&\alpha \left(
p_{2}+q_{2}\right) +\frac{\alpha \left( \alpha -1\right) }{2}\left(
p_{1}^{2}+q_{1}^{2}\right) \\
&=&\alpha \left( p_{2}+q_{2}\right) +\frac{\left( \alpha -1\right) }{\alpha }%
\left( \lambda +\mu \right) ^{2}a_{2}^{2}.
\end{eqnarray*}%
Therefore, we have 
\begin{equation}
a_{2}^{2}=\frac{\alpha ^{2}\left( p_{2}+q_{2}\right) }{\left( \lambda +\mu
\right) ^{2}+\alpha \left( \mu +2\lambda -\lambda ^{2}\right) }.
\label{eq14}
\end{equation}%
Applying Lemma \ref{l1} for (\ref{eq14}), we obtain%
\begin{equation*}
\left\vert a_{2}\right\vert \leq \frac{2\alpha }{\sqrt{\left( \lambda +\mu
\right) ^{2}+\alpha \left( \mu +2\lambda -\lambda ^{2}\right) }}
\end{equation*}%
which gives us desired estimate on $\left\vert a_{2}\right\vert $ as
asserted in (\ref{eq5}).

Next, in order to find the bound on $\left\vert a_{3}\right\vert ,$ by
subtracting (\ref{eq11}) from (\ref{eq10}), we get%
\begin{equation}
2\left( 2\lambda +\mu \right) a_{3}-2\left( 2\lambda +\mu \right)
a_{2}^{2}=\alpha p_{2}+\frac{\alpha \left( \alpha -1\right) }{2}%
p_{1}^{2}-\left( \alpha q_{2}+\frac{\alpha \left( \alpha -1\right) }{2}%
q_{1}^{2}\right) .  \label{eq15}
\end{equation}%
It follows from (\ref{eq12}), (\ref{eq13}) and (\ref{eq15}) that%
\begin{equation}
a_{3}=\frac{\alpha ^{2}\left( p_{1}^{2}+q_{1}^{2}\right) }{2\left( \lambda
+\mu \right) ^{2}}+\frac{\alpha \left( p_{2}-q_{2}\right) }{2\left( 2\lambda
+\mu \right) }  \label{eq16}
\end{equation}%
Applying Lemma \ref{l1} for (\ref{eq16}), we readily get%
\begin{equation*}
\left\vert a_{3}\right\vert \leq \frac{4\alpha ^{2}}{\left( \lambda +\mu
\right) ^{2}}+\frac{2\alpha }{2\lambda +\mu }.
\end{equation*}%
This completes the proof of Theorem \ref{t1}.
\end{proof}

If we take $\mu =1$ in Theorem \ref{t1}, we have the following corollary.

\begin{corollary}
\label{c1}\cite{Fra} Let $f(z)$ given by (\ref{eq1}) be in the class $%
\mathcal{N}_{\Sigma }^{1}\left( \alpha ,\lambda \right) ,$ $0<\alpha \leq 1$
and $\lambda \geq 1.$ Then%
\begin{equation*}
\left\vert a_{2}\right\vert \leq \frac{2\alpha }{\sqrt{\left( \lambda
+1\right) ^{2}+\alpha \left( 1+2\lambda -\lambda ^{2}\right) }}\text{ and }%
\left\vert a_{3}\right\vert \leq \frac{4\alpha ^{2}}{\left( \lambda
+1\right) ^{2}}+\frac{2\alpha }{2\lambda +1}.
\end{equation*}
\end{corollary}

If we choose $\lambda =\mu =1$ in Theorem \ref{t1}, we get the following
corollary.

\begin{corollary}
\label{c2}\cite{Sri} Let $f(z)$ given by (\ref{eq1}) be in the class $%
\mathcal{N}_{\Sigma }^{1}\left( \alpha ,1\right) ,$ $0<\alpha \leq 1.$ Then%
\begin{equation*}
\left\vert a_{2}\right\vert \leq \alpha \sqrt{\frac{2}{\alpha +2}}\text{ and 
}\left\vert a_{3}\right\vert \leq \frac{\alpha \left( 3\alpha +2\right) }{3}.
\end{equation*}
\end{corollary}

If we choose $\lambda =\mu +1=1$ in Theorem \ref{t1}, we obtain well-known
the class $\mathcal{N}_{\Sigma }^{0}\left( \alpha ,1\right) =\mathcal{S}%
_{\Sigma }^{\ast }[\alpha ]$ of strongly bi-starlike functions of order $%
\alpha $ and get the following corollary.

\begin{corollary}
\label{c51} Let $f(z)$ given by (\ref{eq1}) be in the class $\mathcal{S}%
_{\Sigma }^{\ast }[\alpha ],$ $0<\alpha \leq 1.$ Then%
\begin{equation*}
\left\vert a_{2}\right\vert \leq \frac{2\alpha }{\sqrt{1+\alpha }}\text{ and 
}\left\vert a_{3}\right\vert \leq \alpha (4\alpha +1)\text{ }.
\end{equation*}
\end{corollary}

\section{Coefficient bounds for the function class $\mathcal{N}_{\Sigma }^{%
\protect\mu }\left( \protect\beta ,\protect\lambda \right) $}

\begin{definition}
\label{d2}A function $f(z)$ given by (\ref{eq1}) is said to be in the class $%
\mathcal{N}_{\Sigma }^{\mu }\left( \beta ,\lambda \right) $ if the following
conditions are satisfied:%
\begin{equation}
f\in \Sigma \text{ and }\Re \left( \left( 1-\lambda \right) \left( \frac{f(z)%
}{z}\right) ^{\mu }+\lambda f^{\prime }(z)\left( \frac{f(z)}{z}\right) ^{\mu
-1}\right) >\beta  \label{eq17}
\end{equation}%
\begin{equation*}
(0\leq \beta <1,\text{ }\mu \geq 0,\text{ }\lambda \geq 1,\text{ }z\in 
\mathcal{U})
\end{equation*}%
and%
\begin{equation}
\Re \left( \left( 1-\lambda \right) \left( \frac{g(w)}{w}\right) ^{\mu
}+\lambda g^{\prime }(w)\left( \frac{g(w)}{w}\right) ^{\mu -1}\right) >\beta
\label{eq18}
\end{equation}%
\begin{equation*}
(0\leq \beta <1,\text{ }\mu \geq 0,\text{ }\lambda \geq 1,\text{ }w\in 
\mathcal{U})
\end{equation*}%
where the function $g$ is defined by (\ref{eq4}).
\end{definition}

The class which is satisfy the conditon (\ref{eq17}) except $f\in \Sigma $
also was studied with other aspects by Zhu \cite{Zhu}.

Note that for $\lambda =\mu =1,$ the class $\mathcal{N}_{\Sigma }^{1}\left(
\beta ,1\right) $ introduced and studied by Srivastava et al. \cite{Sri} and
for $\mu =1,$ the class $\mathcal{N}_{\Sigma }^{1}\left( \beta ,\lambda
\right) $ introduced and worked by Frasin and Aouf \cite{Fra}.

\begin{theorem}
\label{t2}Let $f(z)$ given by (\ref{eq1}) be in the class $\mathcal{N}%
_{\Sigma }^{\mu }\left( \beta ,\lambda \right) ,$ $0\leq \beta <1,$ $\lambda
\geq 1$ and $\mu \geq 0.$ Then%
\begin{equation}
\left\vert a_{2}\right\vert \leq \min \left\{ \sqrt{\frac{4\left( 1-\beta
\right) }{\left( \mu +1\right) \left( 2\lambda +\mu \right) }},\text{ }\frac{%
2\left( 1-\beta \right) }{\lambda +\mu }\right\}  \label{eq19}
\end{equation}%
and%
\begin{equation}
\left\vert a_{3}\right\vert \leq \left\{ 
\begin{array}{cc}
\min \left\{ \frac{4\left( 1-\beta \right) }{\left( \mu +1\right) \left(
2\lambda +\mu \right) },\frac{4\left( 1-\beta \right) ^{2}}{\left( \lambda
+\mu \right) ^{2}}+\frac{2\left( 1-\beta \right) }{2\lambda +\mu }\right\} ;
& 0\leq \mu <1 \\ 
\frac{2\left( 1-\beta \right) }{2\lambda +\mu }; & \mu \geq 1%
\end{array}%
\right. .  \label{eq20}
\end{equation}
\end{theorem}

\begin{proof}
It follows from (\ref{eq17}) and (\ref{eq18}) that there exist $p,q\in 
\mathcal{P}$ such that%
\begin{equation}
\left( 1-\lambda \right) \left( \frac{f(z)}{z}\right) ^{\mu }+\lambda
f^{\prime }(z)\left( \frac{f(z)}{z}\right) ^{\mu -1}=\beta +\left( 1-\beta
\right) p(z)  \label{eq21}
\end{equation}%
and%
\begin{equation}
\left( 1-\lambda \right) \left( \frac{g(w)}{w}\right) ^{\mu }+\lambda
g^{\prime }(w)\left( \frac{g(w)}{w}\right) ^{\mu -1}=\beta +\left( 1-\beta
\right) q(w)  \label{eq22}
\end{equation}%
where $p(z)=1+p_{1}z+p_{2}z^{2}+...$ and $q(w)=1+q_{1}w+q_{2}w^{2}+...$. As
in the proof of Theorem \ref{t1}, by suitably comparing coefficients in (\ref%
{eq21}) and (\ref{eq22}), we get%
\begin{eqnarray}
\left( \lambda +\mu \right) a_{2} &=&\left( 1-\beta \right) p_{1},
\label{eq23} \\
\left( 2\lambda +\mu \right) a_{3}+\left( \mu -1\right) \left( \lambda +%
\frac{\mu }{2}\right) a_{2}^{2} &=&\left( 1-\beta \right) p_{2},
\label{eq24} \\
-\left( \lambda +\mu \right) a_{2} &=&\left( 1-\beta \right) q_{1}
\label{eq25}
\end{eqnarray}%
and%
\begin{equation}
-\left( 2\lambda +\mu \right) a_{3}+\left( 3+\mu \right) \left( \lambda +%
\frac{\mu }{2}\right) a_{2}^{2}=\left( 1-\beta \right) q_{2}.  \label{eq26}
\end{equation}%
Now, considering (\ref{eq23}) and (\ref{eq25}), we obtain%
\begin{equation}
p_{1}=-q_{1}  \label{eq27}
\end{equation}%
and%
\begin{equation}
2\left( \lambda +\mu \right) ^{2}a_{2}^{2}=\left( 1-\beta \right) ^{2}\left(
p_{1}^{2}+q_{1}^{2}\right) .  \label{eq28}
\end{equation}%
Also, from (\ref{eq24}) and (\ref{eq26}), we have%
\begin{equation}
\left( \mu +1\right) \left( 2\lambda +\mu \right) a_{2}^{2}=\left( 1-\beta
\right) \left( p_{2}+q_{2}\right) .  \label{e1}
\end{equation}%
Therefore, from the equalities (\ref{eq28}) and (\ref{e1}) we find that%
\begin{equation*}
\left\vert a_{2}\right\vert ^{2}\leq \frac{\left( 1-\beta \right) ^{2}}{%
2\left( \lambda +\mu \right) ^{2}}\left( \left\vert p_{1}\right\vert
^{2}+\left\vert q_{1}\right\vert ^{2}\right)
\end{equation*}%
and 
\begin{equation*}
\left\vert a_{2}\right\vert ^{2}\leq \frac{\left( 1-\beta \right) }{\left(
\mu +1\right) \left( 2\lambda +\mu \right) }\left( \left\vert
p_{2}\right\vert +\left\vert q_{2}\right\vert \right)
\end{equation*}%
respectively, and applying Lemma \ref{l1}, we obtain%
\begin{equation}
\left\vert a_{2}\right\vert \leq \frac{2\left( 1-\beta \right) }{\lambda
+\mu }  \label{e2}
\end{equation}%
and 
\begin{equation}
\left\vert a_{2}\right\vert \leq 2\sqrt{\frac{1-\beta }{\left( \mu +1\right)
\left( 2\lambda +\mu \right) }}  \label{e3}
\end{equation}%
respectively. If we compare the right sides of the inequalities (\ref{e2})
and (\ref{e3}) we obtain desired estimate on $\left\vert a_{2}\right\vert $
as asserted in (\ref{eq19}).

Next, in order to find the bound on $\left\vert a_{3}\right\vert ,$ by
subtracting (\ref{eq26}) from (\ref{eq24}), we get%
\begin{equation}
2\left( 2\lambda +\mu \right) a_{3}-2\left( 2\lambda +\mu \right)
a_{2}^{2}=\left( 1-\beta \right) \left( p_{2}-q_{2}\right) ,  \label{e4}
\end{equation}%
which, upon substitution of the value of $a_{2}^{2}$ from (\ref{eq28}),
yields%
\begin{equation}
a_{3}=\frac{\left( 1-\beta \right) ^{2}\left( p_{1}^{2}+q_{1}^{2}\right) }{%
2\left( \lambda +\mu \right) ^{2}}+\frac{\left( 1-\beta \right) \left(
p_{2}-q_{2}\right) }{2\left( 2\lambda +\mu \right) }.  \label{eq29}
\end{equation}

Applying Lemma \ref{l1} for (\ref{eq29}), we readily get%
\begin{equation}
\left\vert a_{3}\right\vert \leq \frac{4\left( 1-\beta \right) ^{2}}{\left(
\lambda +\mu \right) ^{2}}+\frac{2\left( 1-\beta \right) }{2\lambda +\mu }.
\label{e41}
\end{equation}

On the other hand, by using the equation (\ref{e1}) in (\ref{e4}), we obtain%
\begin{equation}
a_{3}=\frac{1-\beta }{2(2\lambda +\mu )}\left[ \frac{\mu +3}{\mu +1}p_{2}+%
\frac{1-\mu }{\mu +1}q_{2}\right] .  \label{e5}
\end{equation}%
and applying Lemma \ref{l1} for (\ref{e5}), we get%
\begin{equation}
\left\vert a_{3}\right\vert \leq \frac{1-\beta }{2\lambda +\mu }\left[ \frac{%
\mu +3}{\mu +1}+\frac{\left\vert 1-\mu \right\vert }{\mu +1}\right] .
\label{e6}
\end{equation}

Now, let us investigate the bounds on $\left\vert a_{3}\right\vert $
according to $\mu .$

\textbf{Case1. }We suppose that let $0\leq \mu <1,$ thus from (\ref{e6})%
\begin{equation}
\left\vert a_{3}\right\vert \leq \frac{4\left( 1-\beta \right) }{(2\lambda
+\mu )(\mu +1)}  \label{e8}
\end{equation}%
which is the first part of assertion (\ref{eq20}).

\textbf{Case2. }We suppose that let $\mu \geq 1,$ thus from (\ref{e6}) we
easily see that%
\begin{equation}
\left\vert a_{3}\right\vert \leq \frac{2\left( 1-\beta \right) }{2\lambda +1}
\label{e7}
\end{equation}%
which is the second part of assertion (\ref{eq20}). When are compared the
right sides of inequalities (\ref{e41}) and (\ref{e7}) we see that the right
side of (\ref{e7}) smaller than the right side of (\ref{e41}).

This completes the proof of Theorem \ref{t2}.
\end{proof}

If we write $\mu =1$ in first parts of assertions (\ref{eq19}) and (\ref%
{eq20}) of Theorem \ref{t2}, we have the following corollary.

\begin{corollary}
\label{c3} Let $f(z)$ given by (\ref{eq1}) be in the class $\mathcal{N}%
_{\Sigma }^{1}\left( \beta ,\lambda \right) ,$ $0\leq \beta <1$ and $\lambda
\geq 1.$ Then%
\begin{equation*}
\left\vert a_{2}\right\vert \leq \min \left\{ \sqrt{\frac{2\left( 1-\beta
\right) }{2\lambda +1}},\text{ }\frac{2\left( 1-\beta \right) }{\lambda +1}%
\right\} \text{ and }\left\vert a_{3}\right\vert \leq \frac{2\left( 1-\beta
\right) }{2\lambda +1}.
\end{equation*}
\end{corollary}

If we choose $\lambda =\mu =1$ in first parts of assertions (\ref{eq19}) and
(\ref{eq20}) of Theorem \ref{t2}, we have the following corollary.

\begin{corollary}
\label{c4} Let $f(z)$ given by (\ref{eq1}) be in the class $\mathcal{N}%
_{\Sigma }^{1}\left( \beta ,1\right) ,$ $0\leq \beta <1.$ Then%
\begin{equation*}
\left\vert a_{2}\right\vert \leq \left\{ 
\begin{array}{cc}
\sqrt{\frac{2\left( 1-\beta \right) }{3}}; & 0\leq \beta <\frac{1}{3} \\ 
1-\beta & \frac{1}{3}\leq \beta <1%
\end{array}%
\right. \text{ and }\left\vert a_{3}\right\vert \leq \frac{2\left( 1-\beta
\right) }{3}.
\end{equation*}
\end{corollary}

If we take $\lambda =\mu +1=1$ in Theorem \ref{t2}, we obtain well-known the
class $\mathcal{N}_{\Sigma }^{0}\left( \beta ,1\right) =\mathcal{S}_{\Sigma
}^{\ast }\left( \beta \right) $ of bi-starlike functions of order $\beta $
and get the following corollary.

\begin{corollary}
\label{c5} Let $f(z)$ given by (\ref{eq1}) be in the class $\mathcal{S}%
_{\Sigma }^{\ast }\left( \beta \right) ,$ $0\leq \beta <1.$ Then%
\begin{equation*}
\left\vert a_{2}\right\vert \leq \sqrt{2\left( 1-\beta \right) }\text{ and }%
\left\vert a_{3}\right\vert \leq \left\{ 
\begin{array}{cc}
2\left( 1-\beta \right) ; & 0\leq \beta <\frac{3}{4} \\ 
(1-\beta )(5-4\beta ); & \frac{3}{4}\leq \beta <1%
\end{array}%
\right. \text{ }.
\end{equation*}
\end{corollary}

\end{document}